\documentclass{article}
\usepackage{amsfonts}
\usepackage{amsmath}
\usepackage{color}
\usepackage{amssymb}
\usepackage{epsfig}
\RequirePackage{psfrag}
\RequirePackage{graphicx}
\usepackage{multicol}
\usepackage{float}

\usepackage[colorlinks=true]{hyperref}
\hypersetup{urlcolor=blue, citecolor=blue}
\usepackage[top = 2cm, left=3cm]{geometry}

\newtheorem{theorem}{\bf Theorem}[section]
\newtheorem{lemma}[theorem]{\bf Lemma}
\newtheorem{proposition}[theorem]{\bf Proposition}

\newtheorem{definition}[theorem]{\bf Definition}

\newtheorem{rmk}[theorem]{{\bf Remark}.\it}{}
\newenvironment{keywords}{\small {\bf key words}.\it}{\vskip 10pt}
\newenvironment{AMS}{\small {\bf AMS subject classification}.\it}{\vskip 10pt}
\newenvironment{proof}{{\it Proof: }}{\hfill $\square$}

\def\O{\Omega}
\def\eps{\varepsilon}
\def\OiT{\Omega_i\times(0,T)}

\def\Dd{\mathcal{D}}

\def\phii{\varphi_i}
\def\ds{\displaystyle}
\def\R{\mathbb{R}}
 \def\N{\mathbb{N}}
\def\l{\lambda}
 \def\a{\alpha}
 \def\b{\beta}
 
 \def\beqn{\begin{eqnarray}}
 \def\eeqn{\end{eqnarray}}
 \def\be{\begin{equation}}
 \def\ee{\end{equation}}
 \def\nn{\nonumber}
 \def\refe#1{(\ref{#1})}
\def\ue{u^\eps}

\def\iint{\int\!\!\!\int}
 \def\Mm{\mathcal{M}}
 
 \def\Ee{\mathcal{E}}
 \def\Uu{\mathcal{U}}
 \def\Pih{{\refe{eq_BL_i}-\refe{trans_BL_i}}}
  \def\Pip{{\refe{eq_tot_i}-\refe{trans_flux}-\refe{trans_p}}}
 \def\k{\kappa}
 \def\sign{\mathrm{sign}}
 \def\t{\tilde}
 
 \def\s{\sigma}

\def\Cc{\mathcal{C}}

\begin{document}
\author{Cl\'ement  Canc\`es\thanks{UPMC Univ Paris 06, UMR 7598, Laboratoire Jacques-Louis Lions, F-75005, Paris, France (\href{mailto:cances@ann.jussieu.fr}{\tt cances@ann.jussieu.fr})}
\thanks{The author is partially supported by GNR MoMaS}
}
\title{Asymptotic behavior of two-phase flows in heterogeneous porous media for capillarity depending only on space.\\ {I}. {C}onvergence to the optimal  entropy solution.}
\maketitle

\begin{abstract}
We consider an immiscible two-phase flow in a heterogeneous one-dimensional porous medium. We suppose particularly that the capillary pressure field is discontinuous with respect to the space variable. The dependence of the capillary pressure with respect to the oil saturation is supposed to be weak, at least for saturations which are not too close to $0$ or $1$. 
We study the asymptotic behavior when the capillary pressure tends to a function which does not depend on the saturation. In this paper, we show that if the capillary forces at the spacial discontinuities are oriented in the same direction that the 
gravity forces, or if the two phases move in the same direction, then the saturation profile with capillary diffusion converges toward the unique optimal entropy solution to the hyperbolic scalar conservation law with discontinuous flux functions.
\end{abstract}
%
%
\section{Presentation of the problem}\label{intro}

The resolution of multi-phase flows in porous media are widely used in oil engineering to predict 
the motions of oil in subsoil. Their mathematical study is however difficult, then some physical
assumptions have to be done, in order to get simpler problems 
(see e.g. \cite{AS79,Bear72,GMT96}). A classical simplified model, so-called 
\emph{dead-oil} approximation, consists in assuming that there is no gas, i.e.  
that the fluid is composed of two immiscible and incompressible phases and in neglecting all the different chemical 
species. The oil-phase and the water-phase are then both made of only one component.
\subsection{the dead-oil problem in the one dimensional case}
Suppose that $\R$ represents a one dimensional homogeneous porous medium, with porosity $\phi$ (which
is supposed to be constant for the sake of simplicity). 
If $u$ denotes the saturation of the water phase, and so $(1-u)$ the saturation of the 
oil phase
thanks to the dead-oil approximation, 
writing the volume conservation of each phase leads to:
\beqn
&\ds \phi \partial_t u + \partial_x V_w = 0,& \label{cons_w_1}\\
& -\ds \phi \partial_t u +  \partial_x V_o = 0, & \label{cons_o_1}
\eeqn
where $V_o$ (resp. $V_w$) is the filtration speed of the oil phase (resp. water phase). 
Using the empirical diphasic Darcy law, we claim that
\be\label{Darcy}
 V_\b= -K \frac{k_{r,\b}(u)}{\mu_\b} \left( \partial_x P_\b -\rho_\b {g} \right),  \qquad \b=o,w,
\ee
where $K$ is the global permeability, only depending on the porous media, $\mu_\b, P_\b, \rho_\b$ are 
respectively  the dynamical viscosity, the pressure and the density of the phase $\b$, ${g}$ 
represents the effect of gravity, $k_{r,\b}$ denotes the relative permeability of the phase $\beta$.
This last term comes from the interference of the two phases in the porous media. 
\vskip 5pt
There exists $s_\star\in [0,1)$ such that the function 
$k_{r,w}$ is non-decreasing, with $k_{r,w}(u)=0$ if $0\le u\le s_\star < 1$, 
and $k_{r,w}$ is increasing on $[s_\star,1]$. 
The function $k_{r,o}$ is supposed to be non-increasing, with $k_{r,o}(1)=0$.
We suppose that there exists $s^\star\in(s_\star,1]$ such that $k_{r,o}(s)=0$ 
for $s\in[s^\star,1)$, and $k_{r,o}$ is decreasing on $[0,s^\star)$.
\vskip 5pt
The pressures are supposed to be linked by the relation
\be\label{P_cap}
P_{cap} (u)=P_w-P_o,
\ee
where $P_{cap}$ is a smooth non-decreasing function called \emph{capillary pressure}.
\vskip 5pt
Adding \refe{cons_w_1} and \refe{cons_o_1}, and using \refe{Darcy} and \refe{P_cap} yields
$$
-\partial_x\left(
\sum_{\b=o,w} K \frac{k_{r,\b}(u)}{\mu_\b} \left( \partial_x P_\b -\rho_\b {\bf g} \right)
\right)=0,
$$
and thus there exists $q$, called total flow-rate,  only depending on time,  such that
\be\label{q}
-\sum_{\b=o,w} K \frac{k_{r,\b}(u)}{\mu_\b} \left( \partial_x P_\b -\rho_\b {g} \right)=q.
\ee
Using \refe{Darcy}, \refe{P_cap} and \refe{q}, \refe{cons_w_1} can be rewritten 
\beqn
&\ds \phi \partial_t u + \partial_x\left( \frac{q k_{r,w}(u)}{k_{r,w}(u)+\frac{\mu_w}{\mu_o}k_{r,o}(u)}\right) &\nn\\
-& \ds K \partial_x \left(\frac{k_{r,w}(u)k_{r,o}(u)}{\mu_o k_{r,w}(u)+\mu_w k_{r,o}(u)}\left(
\partial_xP_{cap}(u)-(\rho_w-\rho_o){g}\right)\right)=0. &\label{cons_w_2}
\eeqn
Supposing that the total 
flow rate $q$ does not depend on times, and after a convenient 
rescaling, equation  \refe{cons_w_2} becomes 
\be\label{eq_h}
\partial_t u + \partial_x( f(u) - \l (u) \partial_x \pi(u))=0,
\ee
where $f$ is a Lipschitz continuous function, fulfilling $f(0)=0$, $f(1)=q$, $\l$  is a nonnegative Lipschitz 
continuous functions, 
with $\l(0)=\l(1)=0$, and $\pi$ is a non-decreasing
function, also called capillary pressure. 
The effects of capillarity are often neglected, particularly in the case of reservoir simulation, and so \refe{eq_h} turns 
to a nonlinear hyperbolic equation called Buckley-Leverett equation, and we have to consider 
the initial-value problem
\be\label{BL}\tag{$\mathcal{BL}$}
\left\{\begin{array}{l}
\partial_t u + \partial_x f(u) =0, \\
u(0)=u_0.
\end{array}
\right.
\ee
 \subsection{discontinuous flux functions and optimal entropy solution}
We now consider heterogeneous one dimensional porous media, i.e.  an apposition of several 
homogeneous porous media with different physical properties. This leads to discontinuous 
functions with respect to the spatial variable. For the sake of simplicity, we assume that the heterogeneous 
porous medium is made of only two homogeneous porous media 
represented by the open subsets $\O_1=\R_-^\star$ and $\O_2=\R_+^\star$. Keeping the notations 
of \refe{cons_w_2}, $\phi, K, k_{r,\b}(u,\cdot)$ and $\pi(u,\cdot)$ are now discontinuous functions, i.e.  piecewise constant functions, denoted $\phi_i, K_i, k_{r,\b,i}$ and $\pi_i$ in $\O_i$.
Thus the problem becomes
\be\label{eq_tot_i}
\left\{\begin{array}{l}
\partial_t u + \partial_x( f_i(u) - \l_i (u) \partial_x \pi_i(u))=0,\\
u(0)=u_0, \\
+ \textrm{ transmission condition at }x=0,
\end{array}\right.
\ee
where 
$f_i$ are Lipschitz continuous functions on $[0,1]$, and can be decomposed in the following way:
\be\label{decomp_f}
f_i(u)= q r_i(u) + \l_i(u) (\rho_w-\rho_o){g},
\ee
where $r_i$ is a non-decreasing Lipschitz continuous function fulfilling $r_i(0)=0$, $r_i(1)=1$, 
and $\l_i$ is a non-negative Lipschitz continuous function fulfilling $\l_i(0)=0$, $\l_i(1)=0$. 
We stress here the fact that $q$ and $(\rho_w-\rho_o){g}$ neither depend on the subdomain $i$ nor on time.
\vskip 5pt
We now have to give more details on this transmission conditions at $x=0$.
First neglect the effects of capillarity, so that \refe{eq_tot_i} becomes the apposition of two 
Buckley-Leverett equations, linked by a transmission condition.
\be\label{eq_BL_i}
\left\{\begin{array}{l}
\partial_t u + \partial_x f_i(u) =0,\\
u(0)=u_0, \\
+ \textrm{ transmission condition at }x=0,
\end{array}\right.
\ee
We ask the conservation of mass at the interface between the two porous media, 
then we have to connect the flux. Denoting $u_i$ the trace (if it exists) of $u_{|\O_i}$ on 
$\{x=0\}$, this means that the following Rankine-Hugoniot condition has to be fulfilled:
\be\label{trans_BL_i} 
f_1(u_1)=f_2(u_2).
\ee
Some assumptions has to be done on the flux functions $f_i$ in order to carry out the study. Firstly, we 
suppose that the total flow-rate $q$ is a non-negative constant. Dealing with non-positive $q$ is also possible, since it suffices to change $x$ by $-x$ and $u$ by $(1-u)$. Secondly, we suppose that each $f_i$ has a simple dynamic on $[0,1]$. More precisely, 
\be\label{croissance_hyp2}
\exists b_i \in [0,1) \textrm{ s.t. } f_i \textrm{ is decreasing on} (0,b_i) \textrm{ and increasing on } (b_i,1).  
\ee
With Assumption \refe{croissance_hyp2}, we particularly ensure that 
$$
q = f_i(1)=\max_{s\in[0,1]}(f_i(s)).
$$
The physical meaning of~\eqref{croissance_hyp2} is that buoyancy works on the oil-phase 
in the sense of decaying $x$. The case $b_i=0$ can also correspond to situations where the total flow rate $q$ is sufficiently strong for ensuring that both phases always move in the same direction. Indeed, 
The oil-flux, given in $\O_i$ by $f_i(u)$ has the same sign as the water-flux, given by $q-f_i(u)$. 
Note that the assumption on the dynamic on $f_i$ is often fulfilled by the physical models, as it is stressed in \cite{AJV04} (see also \cite{EGV03}). 
\vskip 5pt

Thirdly, we assume
\begin{equation}\label{Hyp:genuine}
f_1 \textrm{ and } f_2 \textrm{ are not linear on any non-degenerate interval of }(0,1).
\end{equation}
\begin{figure}[htb]
\centering
\includegraphics[width=6cm]{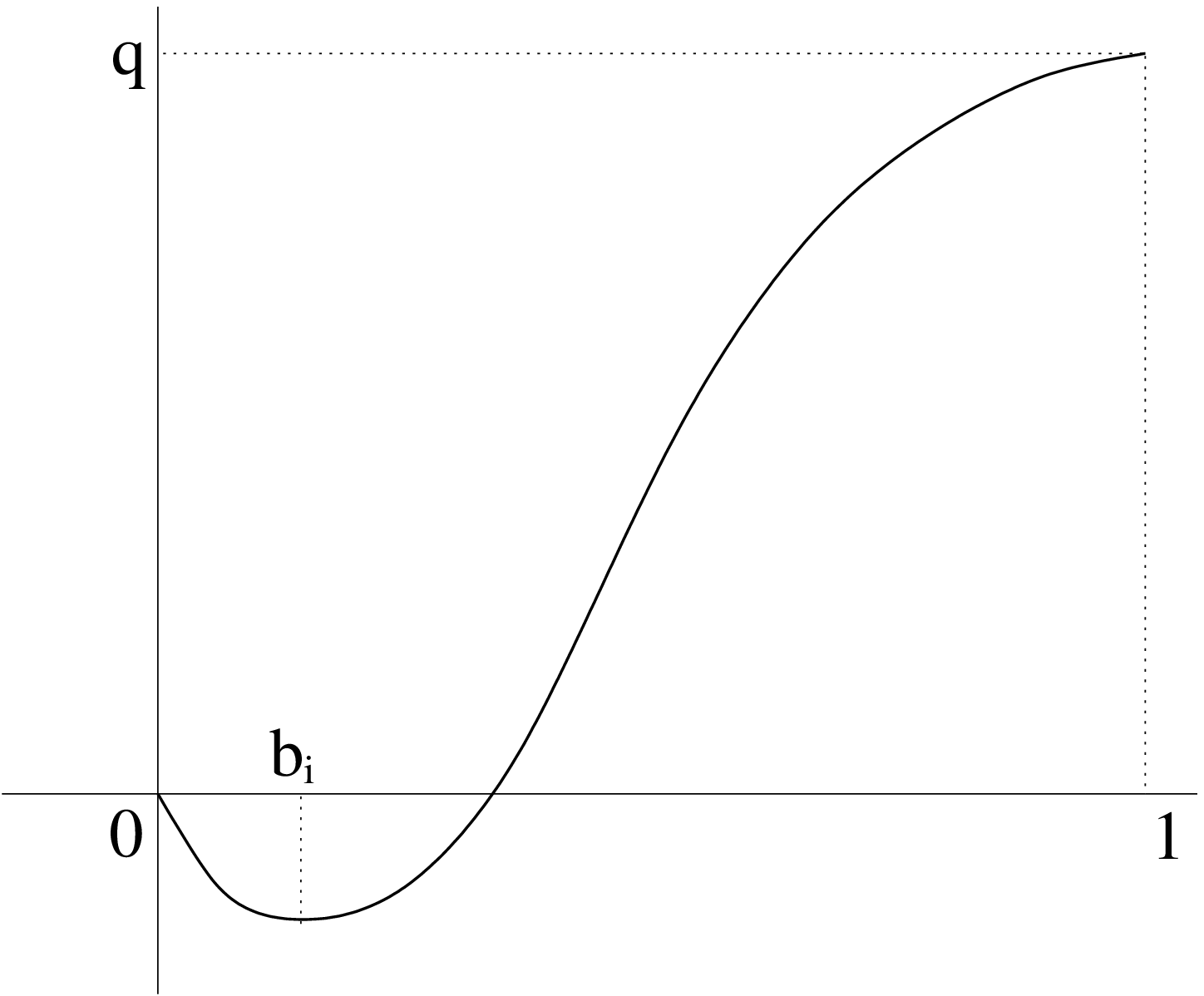}
\caption{example of $f_i$ fulfilling \refe{croissance_hyp2} and \refe{Hyp:genuine}}
\label{f_i} 
\end{figure}
This latter assumption allows us to claim, thanks to \cite{Pan07} (see also \cite{Vas01}) that a solution of 
$$
\partial_t u + \partial_x f_i(u) = 0
$$
satisfying the entropy inequalities in $\O_i\times(0,T)$: $\forall \k \in [0,1]$
\begin{equation}\label{eq:entro_i}
\partial_ t | u-\k | + \partial_x \left( \sign(u-\k) (f_i(u) - f_i(\k) ) \right)\le 0 \qquad \textrm{ in } \Dd'(\O_i\times[0,T))\end{equation}
admits a strong trace $u_i$ on $\{x=0\}\times(0,T)$. 
\vskip 5pt
\begin{rmk}\label{rmk:trace}
As it is proven in Section~\ref{aeconv}, if $u_0\in BV(\R)$, the solution $u$ we consider has strong traces on the interface without assuming~\eqref{Hyp:genuine}. The assumption~\eqref{Hyp:genuine} provides the existence of strong traces for initial data in $L^\infty(\R)$.
\end{rmk}
\vskip 5pt
The problem \refe{eq_BL_i}-\refe{trans_BL_i} has been widely studied recently (see e.g. 
\cite{AJV04, AMV07, AP05,Bac04,Bac06,BV05,CEK08, Die95, Die96, DW96, GNPT07,Jim07, Kaa99, KRT02a,KRT02b,KRT03,Pan09, SV03,Tow00,Tow01}). It has been particularly shown by Adimurthi, Mishra and Veerappa Gowda \cite{AMV05} that there are infinitely many solutions satisfying~\eqref{eq:entro_i}. Additional entropy conditions has to be considered at the interface $\{x=0\}$. We refer to \cite{AMV05} and \cite{BKT09} for a detailed discussion on the possible choices of entropy conditions at the interface. According to Kaasschieter~\cite{Kaa99} and Adimurthi, Jaffr\'e and Veerappa Gowda \cite{AJV04}, the relevant entropy condition at the interface for two-phase flows with \emph{continuous capillary pressure field} is the so-called \emph{optimal entropy condition} introduced in \cite{AMV05}.
Assuming that both $f_i$ are convex, the optimal entropy solution is characterized as follows: the discontinuity at the interface between $u_1$ and $u_2$ can not be undercompressive:
\begin{equation}\label{eq:undercomp}
\min\left\{ 0, f_1'(u_1) \right\} \max\left\{ 0, f_2'(u_2) \right\} = 0.
\end{equation}
Following the idea of Audusse and Perthame \cite{AP05}, the entropy condition at the interface can be derived by comparing the solution to steady states. 
Denoting by 
$\tilde{\k}(x) = \k_i$ if $x\in \O_1$, where $f_1(\k_1) = f_2(\k_2)$, then $\tilde\k$ has to satisfy the relation
$$
\min\left\{ 0, f_1'(\k_1) \right\} \max\left\{ 0, f_2'(\k_2) \right\} = 0.
$$
We denote by $\k_{\rm opt}(x)$ be the steady state corresponding to the optimal entropy connection appearing  in the work of Adimurthi, Mishra and Veerappa Gowda \cite{AMV05}, described on Figure~\ref{fig:optimal_connection} and defined as follows: 
\begin{itemize}
\item if $f_1(b_1) \le f_2(b_2)$, then 
\be\label{eq:k_opt_1}
\k_{\rm opt}(x) = \left\{ \begin{array}{l}
b_2 \textrm{ if }x>0, \\
\underline{b}_2=\min\left\{ \nu \ |Ê\ f_1(\nu) = f_2(b_2) \right\} \textrm{ if }Êx<0;
\end{array}\right.
\ee
\item if $f_1(b_1) \ge f_2(b_2)$, then 
\be\label{eq:k_opt_2}
\k_{\rm opt}(x) = \left\{ \begin{array}{l}
b_1 \textrm{ if }x<0, \\
\overline{b}_1= \max\left\{ \nu \ |Ê\ f_2(\nu) = f_1(b_1) \right\} \textrm{ if }Êx>0.
\end{array}\right.
\ee
\end{itemize}
It is easy to check that in both cases, the function $(x,t)\mapsto \k_{\rm opt}(x)$ is a steady entropy solution in the sense of Definition~\ref{entro_Pih_def}.
\begin{figure}[htb]
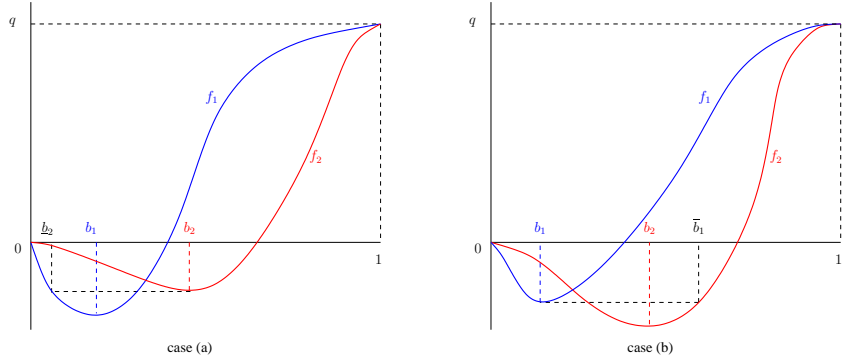

\begin{center}
\resizebox{5cm}{!}{\input{f12.pstex_t}}
\hskip 1cm
\resizebox{5cm}{!}{\input{f15.pstex_t}}
\end{center}
\caption{
We exhibit particular steady states, which are so called \emph{optimal connections} in the work of Adimurthi, Mishra and Veerappa Gowda \cite{AMV05}.
The case (a) corresponds to the optimal connection \eqref{eq:k_opt_1} while the case (b) corresponds to  the optimal connection~\eqref{eq:k_opt_2}.
}
\label{fig:optimal_connection}
\end{figure}
According to \cite{AMV05} and \cite{BKT09}, the condition~\eqref{eq:undercomp} can be replaced by the formulation: $\forall \psi \in \Dd^+(\R\times[0,T))$, 
\beqn
\lefteqn{\int_0^T \hspace{-5pt} \int_\R | u - \k_{\rm opt}| \partial_t \psi \,dxdt 
+ \int_\R | u_0 - \k_{\rm opt}| \psi(\cdot,0) \, dx} \nn\\
&&+ \int_0^T \hspace{-5pt} \sum_{i=1,2} \int_{\O_i} \hspace{-5pt}\sign(u-\k_{\rm opt}) \left( f_i(u)- f_i(\k_{\rm opt})\right) \partial_x \psi\, dxdt \ge 0. \label{eq:comp_opt}
\eeqn
This interface entropy condition does not require anymore that both $f_i$ are convex to provide a $L^1$-contraction semi-group, as it will be stated in Theorem~\ref{th_Bac} and shown in \cite{BKT09}.
\vskip 5pt
\begin{definition}[entropy solution to \Pih]
\label{entro_Pih_def}
Let $u_0 \in L^\infty(\R)$ with $0\le u_0\le 1$, and let $T>0$. 
A function $u$ is said to be an entropy solution to \Pih~if
\begin{enumerate}
\item $u\in L^\infty(\R\times(0,T))$, $0\le u\le 1$ a.e.; 
\item for all $\psi\in \Dd^+(\R\times[0,T))$,
\be\label{eq:weak}
 \int_0^T {\hspace{-5pt}} \int_\R u \partial_t \psi\,dxdt + \int_\R u_0 \psi(\cdot, 0) dx
 + \sum_{i=1,2}  \int_0^T {\hspace{-5pt}} \int_{\O_i} f_i(u) \partial_x \psi \, dxdt =0;
\ee
\item for $i=1,2$, for all $\psi\in \Dd^+(\O_i\times[0,T))$, for all $\k\in [0,1]$, 
\beqn
&\ds \int_0^T {\hspace{-5pt}} \int_{\O_i} |u-\k| \partial_t \psi\, dxdt+ \int_{\O_i} |u_0-\k|\psi(\cdot,0) dx&\nn\\
&\ds +  \int_0^T {\hspace{-5pt}} \int_{\O_i} \sign(u-\k)(f_i(u)-f_i(\k))\partial_x\psi\, dxdt \ge 0;&\label{entro_Pih_for}\eeqn
\item the inequalities~\eqref{eq:comp_opt} hold.
\end{enumerate}
\end{definition}
\vskip 5pt
In the following theorem, we claim the existence and the uniqueness of the entropy solution to the problem~{\Pih}.
\vskip 5pt
\begin{theorem}[Existence and uniqueness of the entropy solution]\label{th_Bac}
Let $u_0 \in L^\infty(\R)$ with $0\le u_0\le 1$, and let $T>0$, then there exists a unique entropy 
solution to \Pih~in the sense of Definition~\ref{entro_Pih_def}. \\
Furthermore, the function $u$ can be supposed to belong to $C([0,T];L^1_{loc}(\R))$, 
and if $u,v$ are two entropy solutions associated to initial data $u_0,v_0$, then, for all 
$R>0$, the following comparison principle holds: $\forall t\in [0,T]$,
 $$
 \int_{-R}^R (u(x,t)-v(x,t))^\pm dx \le  \int_{-R-Ct}^{R+Ct} (u_0(x)-v_0(x))^\pm dx,
$$
where $C=\max_i(Lip(f_i))$, with $Lip(f_i)=\sup_{s\in(0,1)} \left| f_i'(s) \right|.$
\end{theorem}
\vskip 5pt
The uniqueness and the $L^1-$stability of the entropy solution stated above can be seen as a straightforward generalization of  Theorem 3.1 in \cite{BKT09} to the case where $q\neq 0$. The existence of such an entropy solution is provided in \cite{AJV04} by showing the convergence of  the discrete solution corresponding to a Godunov-type scheme, while a modified Engquist-Osher scheme is considered in~\cite{BKT09}.
\vskip 5pt

In the particular case where $f_i(u)= k_i u(1-u)$ for $k_1, k_2 \in \R_+$, then it is shown in \cite{SV03}~that entropy solutions can be obtained as the limit for $\mu\to 0$ and $\delta\to 0$ of the solutions $u^{\mu,\delta}$ to the problem 
$$
\partial_t u^{\mu,\delta} + \partial_x \left( k^\delta(x) u^{\mu,\delta} ( 1 -  u^{\mu,\delta} ) \right) = \mu \partial_{xx}  u^{\mu,\delta}
$$
where $k^\delta$ is a smooth approximation of the piecewise constant function defined by $k(x)=k_i$ if $x\in \O_i$.
\vskip 5pt
It has been proven in \cite{Kaa99} that the entropy solution can also be obtained as limit for $\eps$ tends to $0$ of weak solutions to regularized problems
$$
\partial_t u^\eps + \partial_x \left( f_i(u^\eps) -  \eps \lambda_i(u^\eps)\partial_x \pi_i(u^\eps) \right) = 0
$$
under the assumption that $\pi_1(0) = \pi_2(0)$ and $\pi_1(1) = \pi_2(1)$. This latter assumptions is relaxed in this paper and in \cite{NC_choc}. 
\vskip 5pt
The Godunov-type scheme proposed by Adimurthi, Jaffr\'e and Veerappa Gowda~\cite{AJV04} uses the fact that the fluxes are given by simple algebraic relations. Indeed, The flux $G_{\rm int}(u_1,u_2)=f_1(u_1) = f_2(u_2)$ at the interface corresponding to traces $u_1,u_2$ is given by:
\be\label{eq:Riemann_interf}
G_{\rm int}(u_1,u_2) = \max\left\{
G_1(u_1,1), G_2(0,u_2)\right\},
\ee
where $G_i$ is the Godunov solver corresponding to $f_i$, that is
$$
G_i(u,v) = \left\{ \begin{array}{lll}
\ds\min_{s\in[u,v]} f_i(s) & \textrm{ if }Ê& u\le v,\\
\ds\max_{s\in[v,u]} f_i(s) & \textrm{ if }Ê& u \ge v.
\end{array}\right.
$$
If the flux at the interface is given by
\be\label{eq:flux_int_1}
G_{\rm int}(u_1,u_2) = f_1(u_1) = G_1(u_1,1),
\ee
then the restriction $u_{|_{\O_1}}$ of $u$ to $\O_1$ is the unique entropy solution to 
\be\label{eq:Pgh}
\left\{\begin{array}{ll}
\partial_t u + \partial_x f_1(u) = 0 & \textrm{ in }\O_1\times(0,T),\\[5pt]
u(0,\cdot) = \gamma& \textrm{ in }(0,T),\\[5pt] 
u(\cdot,0) = u_0 &\textrm{ in }\O_1
\end{array}
\right.
\ee
corresponding to $\gamma=1$.
Recall that the trace on $\{x=0\}$ has to be understood in a weak sense (see \cite{BLN79,MNRR96}). 
Since the solution to~\eqref{eq:Pgh} depends in a non-decreasing way of the prescribed trace $\gamma$, 
we can claim that 
\be\label{eq:charac_1}
\textstyle u_{|_{\O_1}} = \sup_{\gamma\in L^\infty((0,T);[0,1])}\left\{Êv \textrm{ solution to }Ê\eqref{eq:Pgh}
\right\}.
\ee
Similarly, in the case where the flux at the interface is given by 
\be\label{eq:flux_int_2}
G_{\rm int}(u_1,u_2) = f_2(u_2) = G_2(0,u_2),
\ee
the restriction $u_{|_{\O_2}}$ of $u$ to $\O_2$ is characterized by 
\be\label{eq:charac_2}
\textstyle
u_{|_{\O_2}} = \inf_{\gamma\in L^\infty((0,T);[0,1])}\left\{Êv \textrm{ solution to }Ê\eqref{eq:Pdh}
\right\},\ee
where 
\be\label{eq:Pdh}
\left\{\begin{array}{ll}
\partial_t u + \partial_x f_2(u) = 0 & \textrm{ in }\O_2\times(0,T),\\[5pt]
u(0,\cdot) = \gamma& \textrm{ in }(0,T),\\[5pt] 
u(\cdot,0) = u_0 &\textrm{ in }\O_2
\end{array}
\right.
\ee
Since the problem is conservative, in both cases the solution $u$ is entropic in both $\O_i\times(0,T)$, i.e. it satisfies \eqref{eq:entro_i}, and 
minimizes the flux through the interface $\{x=0\}$. It is shown in \cite{NC_choc} (see also \cite{Paris09}) that this characterization still holds, 
but that the different physical assumptions lead to the selection of a solution to \Pih~which is not the entropy solution in the sense of Definition~\ref{entro_Pih_def}.

\subsection{heterogeneities involving discontinuous capillarities}
Let us now come back to the problem \refe{eq_tot_i}. 
Suppose for the sake of simplicity that the functions $\pi_i$ are smooth and increasing on $[0,1]$, 
and that $\l_i(u)>0$ if $0<u<1$. The problem is then a spatial coupling of two parabolic problems, and we 
will need to ask two transmission conditions: one for the trace, and one for the flux. Concerning 
the latter, the conservation of mass yields a relation analogous to \refe{trans_BL_i}, which can 
be written with rough notations:
\be\label{trans_flux}
f_1(u_1)- \l_1(u_1) \partial_x \pi_1(u_1)=f_2(u_2)- \l_2(u_2) \partial_x \pi_2(u_2).
\ee

Let us now focus on the trace condition at the interface. In the case of heterogeneous media, the capillary pressure can be discontinuous at the interface. Numerical schemes for simulating such flows has been proposed in \cite{FVbarriere,EEM06,EMS09,EEN98}. 
It has been shown independently in \cite{BLS09} and \cite{CGP09} 
(but see also \cite{BPvD03} and \cite{EEM06}) 
that the connection of the capillary pressures $\pi_i(u_i)$ has to be done in a graphical sense, 
so that phenomena like oil trapping can appear. Thus we have to define  the monotonous 
graphs $\t\pi_i$.
$$
\t\pi_i(u)=
\left\{
\begin{array}{ll}
\pi_i(u) & \textrm{if } 0<u<1,\\ 
(-\infty, \pi_i(0)] & \textrm{if } u=0,\\ 
\left[\pi_i(1), +\infty\right) & \textrm{if } u=1.
\end{array}
\right.
$$
It is shown in \cite{BLS09} and \cite{CGP09} that a natural way to connect the capillary pressures on the interface 
consists in asking:
\be\label{trans_p}
\t\pi_1(u_1)\cap\t\pi_2(u_2)\neq\emptyset.
\ee

In order to state a convenient definition for the solution of \Pip, we introduce the Kirchhoff transformation 
$\phii(u)=\int_0^u \l_i(s)\pi_i'(s)ds.$
\vskip 5pt
\begin{definition}[bounded flux solution]\label{bounded_def}
Let  $u_0\in L^\infty(\R)$, $0\le u_0 \le 1$, and let $T>0$. 
A function $u$ is said to be a bounded flux solution to \Pip~if it fulfills:
\begin{enumerate}
\item $u\in L^\infty(\R\times(0,T))$, $0\le u\le 1$ a.e., 
\item $\partial_x\phii(u)\in L^\infty(\O_i\times(0,T))$, 
\item $\t\pi_1(u_1)\cap\t\pi_2(u_2)\neq\emptyset$ for a.e. $t\in (0,T)$,
\item $\forall \psi \in \Dd(\R\times[0,T))$, 
\beqn
&\ds \int_0^T \int_\R u(x,t)\partial_t \psi(x,t)dxdt +  \int_\R u_0(x) \psi(x,0)dx&\nn \\
&\ds +  \int_0^T \sum_{i=1,2} \int_{\O_i} \left( f_i(u)(x,t)-\partial_x\phii(u)(x,t)
\right)\partial_x \psi(x,t)dxdt=0. &\label{bounded_for}
\eeqn
\end{enumerate}
\end{definition}
\vskip 5pt
The bounded flux solution are so called since the point 2 of Definition~\ref{bounded_def} 
ensures that the flux $f_i(u)-\partial_x \phii(u)$ remains uniformly bounded.
Such a condition will require assumptions on the initial data $u_0$, as it will be stated 
in the following theorem.
\vskip 5pt
\begin{theorem}[existence of a bounded flux solution]\label{bounded_existence}
Let $f_1,f_2$ be Lipschitz continuous functions, and $\varphi_1, \varphi_2$ 
be increasing Lipschitz continuous functions. 
Let $u_0\in L^1(\R)$, $0\le u_0 \le 1$ fulfilling $\partial_x \phii(u_0)\in L^\infty(\O_i)$, 
and $\t\pi_1(u_{0,1})\cap\t\pi_2(u_{0,2})\neq\emptyset$, where $u_{0,i}$ denotes the trace 
on $\{x=0\}$ of ${u_0}_{|\O_i}$. Then there exists a bounded flux solution. 
Moreover, $u$ belongs to $\Cc([0,T];L^1(\R))$.
\end{theorem} 
\vskip 5pt
The first part of this theorem is a straightforward adaptation to the case of unbounded domains and 
non-monotonous $f_i$ of a result from \cite{FVbarriere} and \cite{CGP09} (see also \cite{Sch08}).
This is based on a maximum principle on the fluxes $(f_i(u)-\partial_x\phii(u))$. 
This particularly yields:
\be\label{max_flux}
\| f_i(u)-\partial_x \phii(u) \|_{L^\infty(\OiT)}\le \max_{j=1,2}\left(\| f_j(u_0)-\partial_x \varphi_j(u_0) \|_{L^\infty(\O_j)}
\right).
\ee
If $u_0\in L^1(\O)$, then choosing $\psi=\min(1,(1, R-|x|)^+)$ and letting $R$ tend to $\infty$ gives
$u\in L^\infty((0,T);L^1(\R))$. Moreover, thanks to \cite{Cont_L1}, $u$ 
can be supposed to belong 
to $\Cc([0,T];L^1_{loc}(\R))$. Then $u$ belongs to $\Cc([0,T];L^1(\R))$.
\vskip 5pt
The choice of bounded flux solutions instead of more classical weak solution with 
$\partial_x\phii(u)$ only belonging to $L^2((0,T);L^2_{loc}(\overline\O_i))$ has been motivated 
by the fact that it provides a comparison principle.
\vskip 5pt
\begin{proposition}\label{bounded_comp}
Let $u,v$ be two bounded flux solutions in the sense of Definition~\ref{bounded_def} 
associated to initial data $u_0,v_0$. Then, for all $\psi\in \Dd^+(\R\times[0,T))$, 
\beqn
&\ds \int_0^T \int_\R (u(x,t)-v(x,t))^\pm\partial_t \psi(x,t)dxdt +  \int_\R (u_0(x)-v_0(x))^\pm \psi(x,0)dx&\nn \\
&\ds +  \int_0^T \sum_{i=1,2} \int_{\O_i} \sign_\pm(u(x,t)-v(x,t)) \left( f_i(u)(x,t)-f_i(v)(x,t)\right)
\partial_x \psi(x,t)dxdt &\nn\\
&\ds -  \int_0^T \sum_{i=1,2} \int_{\O_i}  \partial_x(\phii(u)(x,t)-\phii(v)(x,t))^\pm \partial_x\psi(x,t)dxdt \ge 0.&
\label{bounded_comp_for}
\eeqn
\end{proposition}
\vskip 5pt
This proposition is not sufficient to claim the uniqueness, but it will be very useful in the sequel.
In order to obtain a uniqueness result, we have to ask furthermore that the initial data belongs to $L^1(\R)$.
\vskip 5pt
\begin{theorem}[uniqueness of bounded flux solution]\label{bounded_unicite}\!
Let $u_0\in L^1(\R)$, $0\le u_0 \le 1$ a.e., with $\partial_x \phii(u_0)\in L^\infty(\O_i)$ and 
$\t\pi_1(u_{0,1}) \cap \t\pi_2 (u_{0,2}) \neq \emptyset $. 
Then there exists a unique bounded flux solution $u\in \Cc([0,T];L^1(\R))$ in the sense of 
definition \ref{bounded_def}.
\end{theorem}
\vskip 5pt
This theorem is a straightforward consequence Proposition~\ref{bounded_comp}. 
Indeed, choosing $\psi=\min(1,(1, R-|x|)^+)$ in \refe{bounded_comp_for}, and letting 
$\R$ tend to $+\infty$ gives the comparison principle: 
$\forall t\in[0,T]$,
\be\label{comp2}
\int_\R (u(x,t)-v(x,t))^\pm dx \le \int_\R (u_0(x)-v_0(x))^\pm dx.
\ee
The uniqueness result follows.
\subsection{capillary pressure independent of the saturation}
In some cases, the dependence of the capillary pressure $\pi_i$ with respect to the saturation seems to be weak, and some numerical simulation consider capillary pressures only depending on the porous medium, but not on the saturation. More precisely, we aim to consider graphs of capillary pressure on the form
\be\label{graph_C}
\t\pi_i(u)=
\left\{
\begin{array}{ll}
P_i & \textrm{if } 0<u<1,\\ 
(-\infty, P_i] & \textrm{if } u=0,\\ 
\left[P_i, +\infty\right) & \textrm{if } u=1,
\end{array}
\right.
\ee
so that the capillary pressure would roughly speaking not depend on $u$. 
\vskip 5pt
If one considers an interface $\{ x =0 \}$ between two $\O_i$, where the $\t\pi_i$ are on the form \refe{graph_C},
we can give an orientation to the interface: the interface is said to be positively oriented if $P_1>P_2$, and 
negatively oriented if $P_1<P_2$. A positively oriented interface involve positive capillary forces, and 
a negatively oriented involve positive capillary forces. The gravity effects are also oriented 
by the sign of $(\rho_w-\rho_o){g}$ in \refe{decomp_f}.
We have to make the assumption that 
\vskip 5pt
\begin{center}
\em ``either the gravity effects and the interface are oriented in the same way, or the 
convective effects are larger than the gravity effects. 
''
\end{center}
\vskip 5pt
Since we have supposed that gravity works in the sense of decaying $x$, we assume in the sequel that 
\be\label{eq:orientation_interface}
P_1 < P_2.
\ee

We build a family of approximate problems \refe{Pe} 
taking into account the capillary pressure: one suppose that 
$\pi_i^\eps(u)=P_i+\eps u$, where $P_i$ is a constant depending only on the homogeneous subdomain
$\O_i$. 
In fact, any $\pi_i^\eps$ converging uniformly to $P_i$ on $[0,1]$ and 
such that $u\mapsto \int_0^u \l_i(s) \left(\pi_i^\eps\right)'(s) ds $ converges uniformly toward $0$ would fit.

\vskip 5pt
Up to a smoothing of the initial data, we obtain a 
resulting  sequence ${(\ue)}_\eps$ 
of bounded flux solutions for a problem of type \Pip. We will show that under Assumptions~\refe{croissance_hyp2}-\eqref{Hyp:genuine}, this sequence tends almost everywhere to the unique entropy solution to \Pih.

\vskip 5pt 
This result has to be compared to the one presented in the associated paper \cite{NC_choc}, where it is shown that if Assumption~\eqref{eq:orientation_interface} does not hold, non-classical shock can occur at the interface, representing \emph{oil-trapping}.

\subsection{organization of the paper}
The paper is organized as follow: section~\ref{Pe_section} is devoted to the study of the approximate 
problem \refe{Pe}. We first smooth the initial data in a convenient way, and then we give a 
$L^2((0,T);H^1(\O_i))$-estimate on the approximate solutions that shows in particular that 
if the approximate solution~$\ue$ converges almost every where towards a function $u$, then $u$ satisfies the points 1,2 and 3 of Definition~\ref{entro_Pih_def}. In order to prove that $\ue$ converges 
almost everywhere, we derive a family $BV$-estimates.  
In order to check that the last point of Definition~\ref{entro_Pih_def} is fulfilled by the limit $u$ of the approximate solutions $\left(\ue\right)_\eps$

%
%
\section{The approximate problems}\label{Pe_section}
In this section we will define the approximate problem \refe{Pe}, and its solution $\ue$. 
We will state a $L^2((0,T);H^1_{loc}(\overline{\O}_i))$-estimate and a family of $BV$-estimates, 
which will be the key points of the proof of convergence of $\ue$ toward a weak solution of 
the problem~\refe{Pe}.
\vskip 5pt
In order to recover a family of entropy inequalities, we will build some steady solutions 
$\k^\eps$ to the problem~\refe{Pe}, and study their limit as $\eps\to 0$. 
This last point will require strongly Assumption~\refe{croissance_hyp2}.

\subsection{smoothing the initial data}
As it has already been stressed in Theorem~\ref{bounded_existence}, we need to assume some 
regularity on the initial data to ensure the existence of a bounded flux solution to problems 
of the type~\Pip.\\
Let $u_0$ belong to $L^\infty(\R)$, with $0\le u_0 \le 1$, we will build a family 
$\left( u_0^\eps \right)_\eps$ of convenient approximate initial data.
\vskip 5pt
\begin{lemma}\label{init_lem}
Let $u_0\in L^\infty(\R)$, $0\le u_0 \le 1$, then there exists a family $\left(u_0^\eps\right)_\eps$ 
of approximate initial data such that:
\begin{itemize}
\item $u_0^\eps \in C_c^\infty(\R^\star)$, 
$0\le u_0^\eps \le 1$, 
\item $u_0^\eps \to u_0$ a.e. in $\R$, $\| \eps \partial_x u_0^\eps \|_\infty \to 0$ as $\eps \to 0$, and  
$\| \eps\partial_x u_0^\eps \|_\infty \le 1$ for all $\eps>0$,
\item If $u_0\in BV(\R)$, then $\| \partial_x u_0^\eps \|_{L^1(\R)} \le  TV(u_0) +4.$
\end{itemize}
\end{lemma}
\vskip 5pt
\begin{proof}
Let $\a>0$, and let $\rho_\a$ be a mollifier with support in $(-\a,\a)$. 
Let 
$v^\a=\left(u_0\circ\chi_{\a<|x|<1/\a}\right)\star \rho_\a$, then it is clear that 
$v^\a\in \Cc_c^\infty(\R^\star)$, and that $v^\a\to u_0$ a.e. in $\R$ as $\a\to 0$.
Choosing $\eps=\min\left(\a, \frac{\min(1,\sqrt{\alpha})}{\|\partial_x v^\a \|_\infty}  \right)$,
and $u_0^\eps=v^\a$
ends the proof of Lemma~\ref{init_lem}.
\end{proof}
\subsection{the problem \refe{Pe}}
Let $P_1, P_2 \in \R$, 
we define the functions $\pi_i^\eps$ by $\pi_i^\eps(u)=P_i+\eps u$, and 
$$ \t\pi_i^\eps (u)=\left\{
\begin{array}{ll}
P_i + \eps u & \textrm{if } 0<u<1,\\ 
(-\infty, P_i] & \textrm{if } u=0,\\ 
\left[ P_i+\eps, +\infty\right) & \textrm{if } u=1.
\end{array}
\right.
$$ 
\begin{figure}[htb]
\begin{center}
\resizebox{6cm}{!}{
\input{graphe_pression.pstex_t}
}
\label{graphe_pression}
\caption{capillary pressures graphs $\t\pi_i^\eps$.}
\end{center}
\end{figure}
If $\eps$ is small, the intersection of the ranges of the functions $\pi_1^\eps$ and $\pi_2^\eps$ is empty, 
and then the graphical relation $\t\pi_1(u_1)\cap \t\pi_2(u_2) \neq \emptyset$ connecting the capillary 
pressures at the interface becomes 
$$
(1-u_1) u_2 =0.
$$

Let $0\le u_0 \le 1$, and  let $\left(\ue_0\right)_\eps$ be built as in Lemma \ref{init_lem},
let $\phii(u)=\int_0^u \l_i(s) ds$.
 The approximate problem is: $\textrm{find } u^\eps \textrm{ s.t.}$
\be\label{Pe}\tag{$\mathcal{P}^\eps$}
\left\{\begin{array}{ll}
\ds \partial_t \ue +\partial_ x\left( f_i(\ue)-\eps \partial_x \phii(\ue)    \right)=0 & \textrm{in } \OiT,\!\!\!\!\!\!\!\!\!\!\!\!\!\!\!
\\
\ds  \t\pi_1^\eps (\ue)(0^-,t)\cap  \t\pi_2^\eps (\ue)(0^+,t) \neq \emptyset  & \textrm{in } (0,T),\\
\ds f_1(\ue)(0^-,t)-\eps \partial_x \varphi_1(\ue)(0^-,t)  =   
f_2(\ue_2)(0^+,t)-\eps \partial_x \varphi_2(\ue)(0^+,t) & \textrm{in } (0,T),\\
\ds \ue(0)=\ue_0 & \textrm{in }\R.
\end{array}\right.
\ee
This  problems is of type \Pip,  then the notion of bounded flux solution is a good frame to solve it.
\vskip 5pt
\begin{definition}[solution to \refe{Pe}] \label{sol_Pe_def}
A function $\ue$ is said to be a (bounded flux) solution to \refe{Pe} if it fulfills 
\begin{enumerate}
\item $\ue \in L^\infty(\R\times(0,T))$, $0\le \ue \le 1$ a.e.,
\item $\partial_x \phii(\ue)\in L^\infty(\O_i\times(0,T))$,
\item $\forall \psi\in \Dd(\R\times[0,T))$, 
\beqn
&\ds \int_0^T \int_\R \ue(x,t)\partial_t \psi(x,t)dxdt +  \int_\R \ue_0(x) \psi(x,0)dx&\nn \\
&\ds +  \int_0^T \sum_{i=1,2} \int_{\O_i} \left( f_i(\ue)(x,t)-\eps \partial_x\phii(\ue)(x,t)
\right)\partial_x \psi(x,t)dxdt=0. &\label{bounded_for_eps}
\eeqn
\end{enumerate}
\end{definition}
\vskip 5pt
We can use Theorem~\ref{bounded_existence} 
to claim that there exists a family $\left(\ue\right)_\eps$ of 
bounded flux solution to \refe{Pe} in the sense of Definition~\ref{sol_Pe_def}. 
Moreover, this family of solution fulfills,  thanks to \refe{max_flux} and Lemma~\ref{init_lem}: for all $\eps >0$,
\be\label{unif_1}
\| \eps \partial_x \phii(\ue) \|_\infty \le (\max_i(Lip(\phii))+\max_i\|f_i \|_{L^\infty(0,1)}). 
\ee
Since $\ue_0$ belongs to $L^1(\R)$, the solution $\ue$ is furthermore unique in $\Cc([0,T],L^1(\R))$ thanks 
to Theorem~\ref{bounded_unicite}. 

\subsection{a $L^2((0,T);H^1_{loc}(\overline\O_i))$-estimate}
All this subsection is devoted to prove the following estimate.
\vskip 5pt
\begin{proposition}\label{L2H1}
Let $K$ be a compact subset of $\overline\O_i$, and let $\ue$ be a solution of 
\refe{Pe} in the sense of Definition \ref{sol_Pe_def},  
then there exists $C$ depending only on $f_i, K,T$ (and not on $\eps$) such that
$$
\sqrt{\eps} \| \phii(\ue) \|_{L^2((0,T);H^1(K))} \le C.
$$
Particularly, this implies that $\eps\partial_x\phii(\ue)\rightarrow 0$ a.e. in $\OiT$ as $\eps\to 0$.
\end{proposition}
\vskip 5pt
\begin{proof}
We fix $\eps>0$. 
Since the functions $\phii^{-1}$ are not Lipschitz continuous, 
the problem \refe{Pe} is not strictly parabolic, and the function $\ue$ is not a strong solution. 
In order to get more regularity on the approximate solution, we regularize the problem by adding 
an additional viscosity $1/n$ ($n\ge1$), so that the so built approximate solution $\ue_n$ 
is regular enough to perform the calculation below.
\vskip 5pt
Let $n>1$, and let $\varphi_{i,n}(u)=\phii(u)+u/n$, and let $\ue_n$ be a bounded flux solution of \refe{Pe} with $\varphi_{i,n}$ 
instead of $\phii$. From \refe{unif_1}, we know that $\partial_x\varphi_{i,n} (\ue_n)$ is uniformly bounded in 
$L^\infty(\OiT)$, and since $\varphi_{i,n}^{-1}$ is a Lipschitz continuous function, one has 
$\partial_x \ue_n \in L^\infty(\OiT)$. The following weak formulation holds: $\forall \psi \in \Dd(\O\times[0,T))$, 
\beqn
&\ds \int_0^T \int_\R \ue_n(x,t)\partial_t \psi(x,t)dxdt +  \int_\R \ue_0(x) \psi(x,0)dx&\nn \\
&\ds +  \int_0^T \sum_{i=1,2} \int_{\O_i} \left( f_i(\ue_n)(x,t)-\eps\partial_x\varphi_{i,n}(\ue_n)(x,t)
\right)\partial_x \psi(x,t)dxdt=0. &\label{bounded_for_en}
\eeqn
Let $(a,b)\subset \O_i$.
Let $\zeta \in \Dd^+((a,b))$, we deduce from \refe{bounded_for_en} that
\beqn
&\ds \Big<\partial_t\ue_n \ |\   \ue_n\zeta^2 {\Big>} = 
\int_0^T\!\!\!\!\int_a^b f_i(\ue_n) \partial_x( \ue_n \zeta^2)dxdt 
- \eps \int_0^T\!\!\!\!\int_a^b \partial_x\varphi_{i,n}(\ue_n) \partial_x(\ue_n\zeta^2 )dxdt,
&\label{NRJ_1}
\eeqn
where $\Big< \cdot | \cdot \Big>$ is the duality bracket between 
$L^2((0,T);H^{-1}(a,b))$ and $L^2((0,T);H^1_0(a,b)).$
Since $\varphi_{i,n}$ is a Lipschitz continuous function with $\left(\| \l_i \|_\infty + 1/n\right)$ as 
Lipschitz constant, one has
\be\label{NRJ_2}
\int_0^T\!\!\!\!\int_a^b \partial_x\varphi_{i,n}(\ue_n) \partial_x(\ue_n)\zeta^2 dxdt \ge 
\frac{1}{\| \l_i \|_\infty+1/n}\int_0^T\!\!\!\!\int_a^b \left(\partial_x\varphi_{i,n}(\ue_n)\right)^2 \zeta^2 dxdt.
\ee
Let $\Phi_i$ be a primitive of $f_i$, then:
\beqn
\int_0^T\!\!\!\!\int_a^b f_i(\ue_n)\partial_x(\ue_n  \zeta^2 ) dxdt &=& 
\int_0^T\!\!\!\!\int_U \partial_x \Phi_i(\ue_n) \zeta^2 dxdt 
+ \int_0^T\!\!\!\!\int_a^b f_i(\ue_n)\ue_n \partial_x \zeta^2 dxdt\nn\\
&=&\int_0^T\!\!\!\!\int_a^b \left[f_i(\ue_n)\ue_n- \Phi_i(\ue_n)\right] \partial_x \zeta^2 dxdt. \label{NRJ_3}
\eeqn
Admit for the moment Lemma~\ref{NRJ_lem}, stated and proven below.
We deduce from \refe{NRJ_1}, \refe{NRJ_2}, \refe{NRJ_3} and  Lemma \ref{NRJ_lem} that 
\beqn
&\ds \frac{\eps}{\| \l_i \|_\infty+1}\int_0^T\!\!\!\!\int_a^b \left(\partial_x\varphi_{i,n}(\ue_n)\right)^2 \zeta^2 dxdt &\nn\\
&\ds\le 
\frac{1}{2}\| \zeta \|^2_\infty  |b-a|   + \int_0^T\!\!\!\!\int_a^b \left|\ue_n\left( 
f_i(\ue_n)-\eps\partial_x\varphi_{i,n}(\ue_n) \right)
- \Phi_i(\ue_n)\right| |\partial_x\zeta^2 | dxdt . & \nn
\eeqn
Using now the fact that $\ue_n$ is a bounded flux solution, we deduce from \refe{max_flux} that 
$\big[\ue_n(f_i(\ue_n)$ $-\eps\partial_x\varphi_{i,n}(\ue_n)) - \Phi_i(\ue_n)\big]$ is uniformly bounded 
independently of $\eps$ and $n$, and so there exists $C$ only depending on 
$f_i$, $|b-a|$, $u_0$ and $\l_i$ such that:
\be\label{NRJ_43}\nn
\eps \int_0^T\!\!\!\!\int_a^b \left(\partial_x\varphi_{i,n}(\ue_n)\right)^2 \zeta^2 dxdt \le C \left(
\|\partial_x\zeta^2 \|_{L^1((0,T);\Mm(\R))} + \|\zeta \|^2_\infty\right).
\ee
This estimate still holds for $\zeta(x,t) = \chi_{(a,b)}(x)$, for all $(a,b)\in {\overline{\O}_i}^2$, so we obtain
\be\label{L2H1n}
\eps \int_0^T\!\!\!\!\int_a^b \left(\partial_x\varphi_{i,n}(\ue_n)\right)^2 dxdt\le C(2T+1).
\ee
Classical compactness arguments provide the convergence, up to a subsequence, of $\left(\ue_n\right)_n$
to a solution $\ue$ of \refe{Pe} in $L^p(\R\times(0,T))$, $1\le p<+\infty$. 
This ensures particularly that, up to a subsequence,
$$
\lim_{n\to +\infty} \ue_n = \ue \textrm{ a.e. in } \R\times (0,T).
$$
Taking the limit w.r.t. $n$ in \refe{L2H1n} yields
$$
\eps \int_0^T\!\!\!\!\int_a^b \left(\partial_x\varphi_{i}(\ue)\right)^2 dxdt\le C(2T+1).
$$
\end{proof}
\vskip 5pt
\begin{lemma}\label{NRJ_lem} 
Let $\ue_n$ be an approximate solution of \refe{Pe} with $\varphi_{i,n}$ instead 
of $\phii$, and let $\psi\in \Dd^+((a,b))$, then
$$
\Big<\partial_t\ue_n \ |\   \ue_n\zeta^2 \Big>
 \ge - \frac{1}{2}\| \zeta \|^2_\infty  |b-a|.
$$
\end{lemma}
\vskip 5pt
\begin{proof}
Since $\left(\ue_n\zeta\right)\in L^2((0,T);H^1_0(a,b))$, and 
$\partial_t \left(\ue_n\zeta\right)\in L^2((0,T);H^{-1}(a,b))$, and so, up to a negligible set,
$(\ue_n\zeta)\in C([0,T];L^2(a,b))$, and 
\beqn
\Big<\partial_t\ue_n \ |\   \ue_n\zeta^2 {\Big>} &=& \Big<\partial_t\ue_n\zeta \ |\   \ue_n\zeta {\Big>} \nn\\
&=&\frac{1}{2}\int_a^b \ue_n (x,T)^2  \zeta^2(x) dx  - \frac{1}{2}\int_a^b u_0(x)^2 \zeta^2(x)dx   \nn \\
& \ge& - \frac{1}{2}\| \zeta \|^2_\infty  |b-a|. \nn
\eeqn
\end{proof}
\subsection{the $BV$-estimates}
 In this section we suppose that $u_0\in BV(\R)$. In order to avoid heavy notations which would not lead to a good comprehension 
 of the problem, the following proof will be formal. To establish the following estimates in a rigorous frame, one can 
introduce of a thin layer $(-\eta,\eta)$ on which the pressure variates smoothly to replace the 
 interface, and add some additional viscosity to obtain smooth strong solutions to the problem.
 This regularization of the problem has been performed in \cite{BPvD03} and in \cite{CGP09}. 
 \vskip 5pt
For $a,b\in[0,1]$, we denote by 
$$F_i(a,b)=sign(a-b)(f_i(a)-f_i(b)).$$
 \begin{lemma}\label{BV_lem0}
 There exists $C$ depending only on $f_i$, $T$, $u_0$ such that
$$
\left| \partial_t F_i(\ue,\k)\right|_{\Mm_b(\OiT)} \le C.
$$
 \end{lemma}
 \begin{proof}
 Suppose in the sequel that $\ue$ is a strong solution, i.e.  
 \be\label{strong}\nn
 \partial_t \ue + \partial_x\left[  f_i(\ue)-\eps \partial_x \phii(\ue) \right] = 0
 \ee
 holds point-wise in $\OiT$. Let $h>0$, and let $t\in (0,T-h)$. Comparing $\ue(\cdot, \cdot+h)$ and $\ue$ 
 with \refe{comp2} yieds
 $$
\int_\R |\ue(x,t+h)-\ue(x,t)|dx \le  \int_\R |\ue(x,h)-\ue_0(x)|dx.
 $$
Dividing by $h$ and letting $h$ tend to $0$, one can claim using the fact that $\ue$ is supposed to be a strong solution
\beqn
\int_\R |\partial_t \ue(x,t) | dx &=& \sum_{i=1,2}\int_{\O_i} |  
\partial_x\left[ f_i(\ue)(x,t)-\eps \partial_x \phii(\ue)(x,t) \right]  | dx\nn \\
& \le & \int_\R |\partial_t \ue_0(x) | dx \nn \\
& \le & \sum_{i=1,2}\int_{\O_i} |  \partial_x\left[ f_i(\ue_0)(x)-\eps \partial_x \phii(\ue_0)(x)  \right]  | dx\nn
\eeqn
Lemma \ref{init_lem} then ensures that there exists $C$ not depending on $\eps$ such that
\be\label{BV1}
\int_0^T\int_\R |\partial_t \ue(x,t) | dxdt = \int_0^T\sum_{i=1,2}\int_{\O_i} |  
\partial_x\left[ f_i(\ue)(x,t)-\eps \partial_x \phii(\ue)(x,t) \right] | dxdt \le C.
\ee
Thanks to the regularity of $f_i$, this particularly ensures that, if we denote 
by $\Mm_b(\OiT)$ the set of the bounded Radon measure on $\OiT$, 
i.e.  the dual space of $\Cc_c(\overline\O_i\times[0,T),\R)$ with the uniform norm, 
we obtain: $\forall \k\in [0,1]$,
$$
\left| \partial_t F_i(\ue,\k)\right|_{\Mm_b(\OiT)} \le C\| f_i' \|_{\infty}.
$$
\end{proof}
\begin{lemma}\label{BV_lem}
There exists $C$ depending only on $u_0$, $f_i$ and $T$ such that 
$$
\Big| \partial_x \Big( F_i(\ue, \k) - \eps \partial_x | \phii(\ue)-\phii(\k) | \Big) \Big|_{\Mm_b(\OiT)} \le C.
$$
\end{lemma}
\vskip 5pt
\begin{proof}
It follows from the work of Carillo (see e.g. \cite{Car99}) that for all $\k\in [0,1]$,  for all $\psi\in \Dd^+(\O_i \times [0,T))$,
\beqn
&\ds \int_0^T\int_{\O_i} | \ue -\k | \partial_t \psi dxdt + \int_{\O_i} | \ue -\k | \psi(0) dx  &\nn\\
&\ds  \int_0^T\int_{\O_i} \Big( F_i(\ue, \k) - \eps \partial_x | \phii(\ue) - \phii(\k) |  \Big) \partial_x \psi dxdt \ge 0.& \label{car1}
\eeqn
Let $\eta>0$, we denote by $\omega_\eta(x)= \left( 1-|x|/\eta \right)^+$, and suppose now that $\psi$ belongs to 
$ \Dd^+(\overline\O_i \times [0,T))$, i.e.  $\psi$ does not vanish on the interface $\{x=0\}$. 
Estimate \refe{car1} still holds when we consider $\psi_\eta=\psi (1-\omega_\eta)$ as test function.
\beqn
&\ds \int_0^T\int_{\O_i} | \ue -\k | (1-\omega_\eta) \partial_t \psi dxdt 
+ \int_{\O_i} | \ue -\k | \psi(0) (1-\omega_\eta) dx  &\nn\\
&\ds +  \int_0^T\int_{\O_i} \Big( F_i(\ue, \k) - \eps \partial_x | \phii(\ue) - \phii(\k) |  \Big)(1-\omega_\eta) \partial_x \psi dxdt & \nn \\
\ge &\ds  \int_0^T\int_{\O_i} \Big( F_i(\ue, \k) - \eps \partial_x | \phii(\ue) - \phii(\k) |  \Big)  \psi \partial_x\omega_\eta dxdt .& \label{car2}
\eeqn
The fact that the flux induced by $\ue$ is uniformly bounded w.r.t. $\eps$ thanks to \refe{unif_1} implies 
that there exists $C$ depending only on $u_0$, $f_i$ such that 
\be\label{unif_car}
\left\| F_i(\ue, \k) - \eps \partial_x | \phii(\ue) - \phii(\k) | \right\|_{L^\infty(\OiT)} \le C.
\ee
Then we obtain the following estimate on the right-hand-side in \refe{car2}:
$$
\left| \int_0^T\int_{\O_i} \Big( F_i(\ue, \k) - \eps \partial_x | \phii(\ue) - \phii(\k) |  \Big)  \psi \partial_x\omega_\eta dxdt \right|\le C T \| \psi \|_\infty.
$$
Letting $\eta$ tend to $0$ in inequality \refe{car2} gives: $\forall \psi\in \Dd^+(\overline\O_i\times [0,T))$, $\forall \k \in [0,1]$,  
\beqn
&\ds \int_0^T\int_{\O_i} | \ue -\k |  \partial_t \psi dxdt + \int_{\O_i} | \ue -\k | \psi(0) dx  &\nn\\
&\ds + \int_0^T\int_{\O_i} \Big( F_i(\ue, \k) - \eps \partial_x | \phii(\ue) - \phii(\k) |  \Big) \partial_x \psi dxdt \ge - CT \| \psi \|_\infty&  \label{car3}
\eeqn
We introduce now a monotonous function $\chi_\psi \in \Dd^+(\overline\O_i)$ equal to $1$ on the support of $\psi(\cdot,t)$ for all $t\in [0,T]$
(so that $\|\partial_x \chi_\psi \|_{L^1(\O_i)}=1$), then 
$\| \psi \|_\infty \chi_\psi \ge \psi $. Choosing $\| \psi \|_\infty \chi_\psi - \psi $ and  $\| \psi \|_\infty \chi_\psi + \psi $
as test function in \refe{car3} yields, using \refe{unif_car} once again
\be\label{BV_car1}
 \left| \partial_t |\ue-\k | - \partial_x \Big( F_i(\ue, \k) - \eps \partial_x | \phii(\ue) - \phii(\k) | \Big) \right|_{\Mm_b(\OiT)}\le 2CT. 
 \ee
Thanks to \refe{BV1}, there exists $C'$ depending only on $u_0$, $f_i$ and $T$ such that
\be\label{BV_car2}
\left| \partial_t |\ue-\k | \right|_{\Mm_b(\OiT)}\le C'.
\ee
Lemma \ref{BV_lem} is so a consequence of \refe{BV_car1} and \refe{BV_car2}.
\end{proof}
\vskip 5pt
\begin{proposition}\label{BV_prop}
Let $u_0\in BV(\R)$ and let $K=[a,b] \subset \overline\O_i$.
We introduce $z_{K}^\eps(x,t)$ defined on the whole space $\R^2$ given by:
$$
z_{K,\k}^\eps(x,t)=\left\{\begin{array}{ll}
F_i(\ue, \k)(x,t) &\textrm{if } (x,t)\in K\times(0,T),\\
0&\textrm{otherwise.}
\end{array}\right.
$$
There exists $C$ depending only on $u_0, f_i, T, K$ 
and a uniformly bounded function $r_{K,\k}$, with $r_{K,\k}(\eps)$ tends uniformly to $0$ with respect to $\k$ as $\eps\to 0$,  such that,
for all $(\xi,h)\in\R^2$, 
$$
\iint_{\R^2} \left| z_{K,\k}^\eps(x+\xi,t+h) - z_{K,\k}^\eps(x,t) \right|dxdt\le C(|\xi|+|h|)+r_{K,\k}(\eps).
$$
\end{proposition}
\vskip 5pt
\begin{proof}
Let $h\in\R$, one has
\beqn
&\ds \iint_{\R^2} \left|z_{K,\k}^\eps(x,t+h) - z_{K,\k}^\eps(x,t) \right| dxdt   \le  \left| \partial_t   z_{K,\k}^\eps \right|_{\Mm_b(\R^2)}|h| &\nn \\
\le  &\ds \left( \left|\partial_t F_i(\ue, \k)\right|_{\Mm_b(\OiT)} +2 \|F_i(\ue, \k)\|_{L^\infty}(T+b-a)\right)|h| &\nn
\eeqn
It follows from Lemma~\ref{BV_lem0} that there exists $C_1$ depending only on $u_0, f_i, K, T, \phii$ such that
\be\label{BV_prop_1}
\iint_{\R^2} \left|z_{K,\k}^\eps(x,t+h) - z_{K,\k}^\eps(x,t) \right| dxdt   \le C_1|h|.
\ee
We also define
$$
q^\eps_{K,\k}(x,t)=\left\{\begin{array}{ll}
\eps\partial_x|\phii(\ue)(x,t)-\phii(\k)| &\textrm{if }(x,t)\in K\times(0,T),\\
0&\textrm{otherwise.}
\end{array}\right.
$$
Proposition \ref{L2H1} ensures that $q^\eps_{K,\k}(x,t)$ converges to $0$ almost everywhere in $\R^2$, and the estimate \refe{unif_car} 
ensures us that $q^\eps_{K,\k}(x,t)$ stays uniformly bounded in $L^\infty(\R^2)$ with respect to $\eps$. \\
Let $\xi\in\R$, then we have:
\beqn
&\ds \ds \iint_{\R^2} \left|z_{K,\k}^\eps(x+\xi,t) - q_{K,\k}^\eps(x+\xi,t) -\left(  z_{K,\k}^\eps(x,t) - q_{K,\k}^\eps(x,t)   \right) \right| dxdt     &\nn\\
\le &\ds    \left( \begin{array}{c}
\ds\left|\partial_x \left(F_i(\ue, \k)-\eps\partial_x|\phii(\ue)(x,t)-\phii(\k)| \right)\right|_{\Mm_b(\OiT)}\\
\ds+ 2 (T+b-a)\left(\left\|F_i(\ue, \k) - \eps\partial_x|\phii(\ue)(x,t)-\phii(\k)| \right\|_{L^\infty} \right)
\end{array}\right)  |\xi |. & \label{BV_prop_2}
\eeqn
Using \refe{unif_car} and Lemma~\ref{BV_lem} in \refe{BV_prop_2} yields that there exists $C_2$ depending only on $u_0, f_i, K, T, \phii$ such that
$$
\iint_{\R^2} \left|z_{K,\k}^\eps(x+\xi,t) - q_{K,\k}^\eps(x+\xi,t) -\left(  z_{K,\k}^\eps(x,t) - q_{K,\k}^\eps(x,t)   \right) \right| dxdt  \le C_2 |\xi|.
$$
This particularly ensures that:
\be\label{BV_prop_3}
\iint_{\R^2} \left|z_{K,\k}^\eps(x+\xi,t) - z_{K,\k}^\eps(x,t) \right| dxdt \le C_2|\xi|+ 2 \|q_{K,\k}^\eps(x,t)\|_{L^1(\R^2)}.
\ee
By choosing $C=\max(C_1,C_2)$, one we deduce from \refe{BV_prop_1} and \refe{BV_prop_3} that
$$
\iint_{\R^2} \left|z_{K,\k}^\eps(x+\xi,t+h) - z_{K,\k}^\eps(x,t) \right| dxdt \le C(|h|+|\xi|)+ 2 \|q_{K,\k}^\eps(x,t)\|_{L^1(\R^2)}.
$$
We conclude the proof of Proposition~\ref{BV_prop} by checking that $ \|q_{K,\k}^\eps(x,t)\|_{L^1(\R^2)}$ 
converges uniformly to $0$ with respect to $\k$ as $\eps$ tends to $0$.
\end{proof}\\
\subsection{some steady solutions}\label{steady_sol}
The spatial discontinuities of the saturation and capillary pressure allow us to consider 
Kru\v{z}kov entropies $|u-\k|$ only for non-negative test functions 
\emph{vanishing on the interface}, i.e.  in $\Dd^+(\R^\star\times[0,T))$. 
This is not enough to obtain the convergence of $\ue$ toward an entropy solution $u$ in the 
sense of Definition~\ref{entro_Pih_def}, and a relation has also to be derived 
at of the interface. 
\vskip 5pt
In order to deal with 
general test functions belonging to $\Dd^+(\R\times[0,T))$, 
we will so have to introduce some approximate Kru\v{z}kov entropies 
$| \cdot- \k^\eps(x)|$, where $\k^\eps$ are steady solutions of \refe{Pe}.
Letting $\eps\to0$, those steady solutions converge to piecewise constant 
function $\t\k^j$ defined below. The functions $| \cdot - \t\k^j(x)|$ correspond to the so called \emph{partially adapted entropies} introduced by Audusse and Perthame \cite{AP05}. We will then be able to compare the limit $u$ of approximate solutions $\ue$ to this limit $\t\k^j$, and then to prove that $u$ is the unique entropy solution.
\vskip 5pt

The building of convenient 
$\k^\eps$ strongly uses Assumption~\refe{croissance_hyp2}. It will be shown in \cite{NC_choc} that if \refe{croissance_hyp2} fails, non classical shocks can occur 
at the interface, and the limit $\t u$ of the approximate solutions $\ue$ is thus not an entropy 
solution.
\vskip 5pt

Recall that $q\ge 0$, $P_1<P_2$, and suppose that $0<\eps < P_2-P_1$. Some 
simple adaptations can be done to cover the case $q < 0$.
The transmission condition 
\refe{trans_p} can be summarized as follow: either $u_1=1$, or $u_2=0$.
\vskip 5pt
We have to introduce the following sets:
$$
\Ee_1=\{ \k_1\ / \exists \k_2 \textrm{ with } f_1(\k_1)=f_2(\k_2)\},
$$
$$
\Ee_2=\{ \k_2\ / \exists \k_1 \textrm{ with } f_1(\k_1)=f_2(\k_2)\}.
$$
It follows from Assumption~\refe{croissance_hyp2} that either $\Ee_1=[0,1]$, or $\Ee_2=[0,1]$, and so 
we are ensured that $\k\in[0,1]$ belongs either to $\Ee_1$, or to $\Ee_2$ 
(or of course to both).
Check also that for all $z\in (0,q]$, it follows from \refe{croissance_hyp2} that there exists a unique 
$\k_{i}(z)$ such that $f_i(\k_i(z))=z$. On the contrary, if $z\in \Ee_i$ with $z\le 0$, $z$ has two antecedents 
through $f_i$. Let $\k$ belong to $\Ee_j$, one denotes 
\be\label{okj}
\overline\k_i^j=\max\{\nu \ | \ f_i(\nu)=f_j(\k)\}
\ee
and 
\be\label{ukj}
\underline\k_i^j=\min\{\nu \ | \ f_i(\nu)=f_j(\k)\}.
\ee

\begin{definition}[reachable steady state]\label{Def:reachable}
A function $\k(x)$ is said to be a \emph{reachable steady state} if there exists a steady entropy solution $\k^\eps(x)$ to the problem~\eqref{Pe} in the sense of Definition~\ref{sol_Pe_def} converging to $\k(x)$ in $L^1(\R)$ as $\eps$ tends to $0$.
\end{definition}
\vskip 10pt
This section is devoted to establish the following proposition, that exhibits all the reachable steady states. In particular, all the steady states that are 
not undercompressible are reachable.
\vskip 5pt
\begin{proposition}\label{steady_prop}
For all $\k \in \Ee_j$, there exists a family of steady solutions ${(\k^\eps)}_\eps$ to the problem 
\refe{Pe} such that 
\be
\k^\eps \to \t\k^j \textrm{ a.e. in } \O_i,
\ee
where $ \t\k^j(x) $ can be chosen between:
\begin{itemize}
\item[i)]
$\t\k^j(x) = \overline\k_1^j$ if $x<0$ and $\underline\k_2^j$  if  $x>0$;
\item[ii)] 
$\t\k^j(x) = \overline\k_1^j$ if $x<0$ and $\overline\k_2^j$  if  $x>0$;
\item[iii)]
$\t\k^j(x) = \underline\k_1^j$ if $x<0$ and $\underline\k_2^j$  if  $x>0$.
\end{itemize}
In particular, $\k_{opt}$ defined in \eqref{eq:k_opt_1}-\eqref{eq:k_opt_2} is a reachable steady state.
\end{proposition}
\vskip 5pt
\begin{proof}
Let $\k \in \Ee_j$. 
\begin{itemize}
\item
If $\k =1$, the three limits $\t\k^j$ are identically equal to $1$, which is a steady solution 
fulfilling Proposition~\ref{steady_prop}. 
\item
We suppose now that $\k<1$, and $f_j(\k)>0$. Even in this case, the three reachable limit are 
the same. Thus we only have to build one sequence of converging steady solutions.
Let $y$ be a solution of:
\be\label{EDO}
\left\{\begin{array}{l}
\ds \frac{\mathrm{d}}{\mathrm{d}x}\varphi_1(y)= f_j(\k)-f_1(y),    \qquad\textrm{ for }x>0,\\[10pt]
y(0)=1.
\end{array}\right.
\ee
The solution $y(x)$ converges to $\overline\k^j_1$ as $x\to +\infty$.
The family ${(\k^\eps)}_\eps$ defined by: $\forall \eps$ 
\be\label{k_eps}
\k^\eps(x)=\left\{\begin{array}{ll}
\overline\k_2^j \textrm{ or } \underline \k_2^j &\textrm{ if }x>0,\\
y(-x/\eps) &\textrm{ if }x<0.
\end{array}
\right.
\ee
fulfills so the conclusion of Proposition~\ref{steady_prop}.
\item Suppose now $f_j(\k)=0$. The solutions $\k^\eps(x)$ built with \refe{EDO}-\refe{k_eps} converges toward
the two reachable steady states $i)$ and $ii)$. One can also choose $\k^\eps(x)=0$, which is of course a steady solution. 
\item It remains the case $f_j(\k)<0$. The solutions $\k^\eps(x)$ built with \refe{EDO}-\refe{k_eps} still converges toward
the two reachable steady states $i)$ and $ii)$.\\
Let $w$ be a solution of 
\be\label{EDO2}\nn
\left\{\begin{array}{l}
\ds \frac{\mathrm{d}}{\mathrm{d}x}\varphi_2(w)=f_2(w) - f_j(\k),    \qquad\textrm{ for }x>0,\\[10pt]
w(0)=0.
\end{array}\right.
\ee
and 
$$\k^\eps(x)=\left\{\begin{array}{ll}
 \underline \k_1^j &\textrm{ if }x<0,\\
w(x/\eps) &\textrm{ if }x>0,
\end{array}
\right.
$$
then $\k^\eps$ converges toward the third reachable steady state.
\end{itemize}
\end{proof}

\section{Convergence toward the entropy solution}\label{conv_section}

\subsection{almost everywhere convergence}\label{aeconv}
We state now a lemma which is an adaptation of Helly's selection theorem criterion, 
which states that if ${(v^\eps)}_\eps$ is a family of measurable functions on an open 
subset $\Uu$ of $\R^k$ ($k\ge 1$), uniformly bounded in $L_{loc}^\infty(\Uu)$ and 
$BV_{loc}(\Uu)$, one can extract a subfamily still denoted by ${(v^\eps)}_\eps$ that converges almost everywhere in $\Uu$, and the limit belongs to $BV_{loc}(\Uu)$.
\vskip 5pt
\begin{lemma}\label{Helly}
Let $\Uu$ be an open subset of $\R^k$ ($k\ge1$).
Let ${(v^\eps)}_\eps$ be a family of functions uniformly bounded in $L^\infty(\Uu)$ with respect to 
$\eps$. 
Let $K$ be a compact subset of $\Uu$. For $\zeta\in \R^k$, we denote by 
$$K_\zeta = \{\  x\in K \ | \ x+\zeta \in K\   \}.$$
One assumes that there exists $C>0$ depending only on $K$ 
(and thus not on $\eps$) and a function $r$ 
fulfilling $\lim_{\eps\to 0} r(\eps)=0$ , such that
for all $\zeta\in \R^k$, 
\be\label{BVeps}
\int_{K_\zeta} |v^\eps(x+\zeta)-v^\eps(x)| dx \le C |\zeta| + r(\eps).
\ee
Then, there exists a sequence ${(\eps_n)}_n$ tending to 0, and $v\in BV_{loc}(\Uu)$ such that 
$$
\lim_{n\to +\infty}v^{\eps_n} = v \qquad\textrm{ a.e. in }\Uu.
$$
\end{lemma}
\begin{proof}
Let $K$ be a compact subset of $\Uu$.
Estimate \refe{BVeps} says, roughly speaking, that $v^\eps$ is almost a $BV$-function on $K$, 
i.e.  as close to $BV(K)$ as wanted, provided that $\eps$ is supposed to be small enough. 
So we will build a family 
${(w^\eps)}_\eps$ of $BV$-functions, which will be close to the family 
${(v^\eps)}_\eps$, at least for small 
$\eps$, and we will show that ${(w^\eps)}_\eps$ admits an adherence value $v$ in $BV(K)$
for the $L^1(K)$-topology, and that this $v$ is also an adherence value of ${(v^\eps)}_\eps$.
Another proof for the a.e. convergence toward a function $v$ can be derived directly 
from Kolmogorov compactness criterion (see e.g. \cite{Bre83}). But the advantage of the 
following method is that it provides directly some regularity on the limit $v\in BV_{loc}(\Uu)$.
\vskip 5pt
Let ${(\rho^\eps)}_\eps$ be a sequence of mollifiers, i.e.  smooth, non negative and compactly supported functions
with support included in the ball of center $0$ and radius $\eps$, 
and fulfilling $\| \rho^\eps \|_{L^1(\R^k)}=1$ for all $\eps >0$. 
We define the smooth functions 
$$w^\eps=\t v^\eps \star \rho^\eps,$$
where $\t v^\eps(x) = v^\eps (x)$ if $x\in \Uu$ and $\t v^\eps(x) = 0$ if $x \in \Uu^c$.
\vskip 5pt
Thanks to the regularity of $w^\eps$, 
\be\label{BVeps_1}
\int_{K_\zeta} |w^\eps(x+\zeta)-w^\eps(x)| dx \le \| \nabla w^\eps \|_{\left(L^1(K)\right)^k} |\zeta|.
\ee
Suppose that $\eps < d(K,\partial \Uu)$ (with the convention $d(K,\emptyset)=\infty$).
Thanks to \refe{BVeps}, we have also
\be\label{BVeps_2}
\int_{K_\zeta} |w^\eps(x+\zeta)-w^\eps(x)| dx \le \int_{\{K_\zeta+\eps\}} 
|v^\eps(x+\zeta)-v^\eps(x)| dx \le C |\zeta| + r(\eps),
\ee
where 
$$\{K_\zeta+\eps\} = \{ \ x \in \Uu \ | \ d(x,K) \le \eps \}.$$
Since $r(\eps)$ tends to $0$ as $\eps\to 0$, this particularly ensures 
$$
\limsup_{\eps \to 0} \| \nabla w^\eps \|_{\left(L^1(\R^k)\right)^k} \le C.
$$
The family ${(w^\eps)}_\eps$ is thus bounded in $BV(K_\zeta)$ 
in the neighborhood of $\eps=0$, and thus, 
thanks to Helly's selection criterion, there exist $v\in BV(K_\zeta)$, 
and ${(\eps_n)}_n$ tending to 0 such that
$$
w^{\eps_n} \to v \quad \textrm{ a.e. in } K \textrm{ as } n \to \infty.
$$
Furthermore, for all $n\in\N^\star$, 
\beqn
\| w^{\eps_n} - v^{\eps_n} \|_{L^1(K)} & \le & \int_{K}  \int_{B(0,\eps_n)}  | v^{\eps_n}(x-y)-v^{\eps_n}(x) |\rho^{\eps_n}(y) dydx \nn\\
&\le & C \eps_n + r(\eps_n). \nn
\eeqn
This ensures that $v^{\eps_n}$ tends also almost everywhere toward $v$ as $n$ tends to $+\infty$.
\end{proof}
\vskip 10pt
Lemma~\ref{Helly} will be used to prove the following convergence assertion.
\vskip 5pt
\begin{proposition}\label{aeconv_prop} 
Suppose that $u_0\in BV(\R)$, and let $\ue$ be a solution to \refe{Pe}.
Up to an extraction, there exists $u\in L^\infty(\R\times(0,T))$, $0\le u \le 1$ a.e. such that
$$
\ue \to u \textrm{ a.e. in } \R\times(0,T).
$$
Furthermore, there exists $u_1,u_2 \in L^\infty(0,T)$, 
such that
\begin{eqnarray*}
&\ds \lim_{\eta \to 0} \frac{1}{\eta} \int_0^T \int_{-\eta}^{0} | u(x,t) - u_1(t) | dxdt =0, &\\
&\ds \lim_{\eta \to 0} \frac{1}{\eta} \int_0^T \int_{0}^{\eta} | u(x,t) - u_2(t) | dxdt =0. & 
\end{eqnarray*}
\end{proposition}
\begin{proof}
Let $K$ be a compact subset of $\overline\O_i$.\\
We define the function $H_i:[0,1]\mapsto \R$ by 
\be\label{H}
H_i(u)=\int_0^1 \left( F_i(u,\s) - f_i(\s) \right)d\s, 
\ee
so that, thanks to Proposition~\ref{BV_prop}, there exists $C$ depending on $u_0, f_i, T, K$, and a function $r$ tending to $0$ as 
$\eps$ tends to 0 such that for all $\xi\in \R$, $h\in (0,T)$, 
\be\label{H2}
 \int_0^{T-h} \int_K  \left|  H_i(u^\eps)(x+\xi,t+h) - H_i(u^\eps)(x,t)\right| dxdt  \le C(|\xi|+|h|)+ r(\eps).
\ee
An integration by parts in \refe{H} yields: $\forall u \in [0,1]$
\beqn
H_i(u) & = & -\int_0^1(\s-b_i) \partial_\s (F_i(u,\s)-f_i(\s))d\s + (2 b_i-1)f_i(u) \nn \\
 & = & 2 \int_0^u (\s-b_i) f_i'(\s) d\s + (2 b_i-1)f_i(u) \label{H3} 
\eeqn
where, thanks to \refe{croissance_hyp2},  $f_i$ is decreasing on $[0,b_i]$ and increasing on $[b_i,1]$.
Using Proposition \ref{BV_prop} with $\k=0$,
\be\label{H4}
 \int_0^{T-h} \int_K  \left|  f_i(u^\eps)(x+\xi,t+h) - f_i(u^\eps)(x,t)\right| dxdt  \le C(|\xi|+|h|)+ r(\eps),
 \ee
 where $C$ and $r$ have been updated.
Denoting by $$A_i(u)= \ds \int_0^u (\s-b_i) f_i'(\s) d\s,$$ we obtain from \refe{H3} and \refe{H4}
\be\label{A}
 \int_0^{T-h} \int_K  \left|  A_i(u^\eps)(x+\xi,t+h) - A_i(u^\eps)(x,t)\right| dxdt  \le C(|\xi|+|h|)+ r(\eps),
 \ee
 with a new update for $C$ and $r$.
Thus we deduce from Proposition~\ref{Helly} that, up to an extraction, $A_i(\ue)$ converges 
almost everywhere toward $\overline A_i \in BV(K \times (0,T))$. 
It follows from \refe{croissance_hyp2}  that $A_i$ is an increasing function, and so 
we obtain the convergence almost everywhere  in $K\times (0,T)$ of $\ue$ toward a measurable function $v$.
\vskip 5pt
Since for all $\eps>0$, $0\le \ue \le 1$, there exists $u\in L^\infty(\R\times(0,T))$, $0\le u \le 1$ such that
$\ue$ converges to $u$ in the $L^\infty(\R\times(0,T))$-$\star$-weak sense, we thus have, 
up to an extraction,
\be\label{aeconvK}
\ue \to u \textrm{ a.e. in } K\times(0,T).
\ee
Since \refe{aeconvK} holds for any compact subset $K$ of $\O_i$, for $i=1,2$, we can claim that, 
up to an extraction, 
$$
\ue \to u \textrm{ a.e. in } \R\times(0,T).
$$
Moreover, since $A_i(u)$ belongs to $BV(\R\times(0,T))$, we can claim that $A_i(u)$ admits 
a strong trace on $\{x=0\}\times(0,T)$ (see for instance \cite{attouch2006}). Using once again the fact 
that $A_i^{-1}$ is a continuous function, we can claim that $u$ admits also a strong trace on 
each side of the interface.
\end{proof}
\subsection{convergence toward the entropy solution}
In this section, it is proven that the limit value $u$ for the family $\left(\ue\right)_{\eps}$ exhibited previously is the entropy solution described in Definition~\ref{entro_Pih_def}.
\vskip 5pt
\begin{proposition}\label{conv_entro}
Let $u_0\in BV(\R)$, $0\le u_0 \le 1$ a.e., and let $(\ue_0)_\eps$ a family of 
approximation of $u_0$ given by Lemma~\ref{init_lem}, 
and let ${(\ue)}_\eps$ be the induced sequence  of  bounded flux solution of \Pip. 
Then under Assumption~\refe{croissance_hyp2}
$$\lim_{\eps \rightarrow 0} \ue = u \textrm{ in } L^p_{loc}(\R\times[0,T]),\quad \forall p\in[1,\infty) $$
where $u$ is the unique entropy solution to \Pih~associated to initial data $u_0$.
\end{proposition}
\vskip 5pt
\begin{proof}
Thanks to Proposition~\ref{aeconv_prop}, we can suppose that there exists $u \in L^\infty(\R\times(0,T))$ such that, up to a subsequence, 
\be\label{eq:aeconv_1}
\ue \to u \textrm{ in } L^1_{loc}(\R\times(0,T)) \textrm{ as } \eps \to 0, 
\ee
then 
$$
f_i(\ue) \to f_i(u)  \textrm{ in } L^1_{loc}(\O_i\times(0,T)) \textrm{ as } \eps \to 0.
$$
Furthermore, thanks to Proposition~\ref{L2H1}, one has
\be\label{eq:aeconv_05}
\eps \partial_x \phii(\ue) \to 0 \textrm{ in } L^{1}_{loc}(\O_i\times(0,T)) \textrm{ as } \eps \to 0.
\ee
As a consequence, letting $\eps$ tend to $0$ in \eqref{bounded_for_eps} provides (recall that $u_0^\eps$ tends to $u_0$ in $L^1_{loc}(\R)$) that $u$ is a weak solution to \Pih, i.e. that it satisfies~\eqref{eq:weak}.
\vskip 5pt
Since for $\eps>0$, $\left(\eps\varphi_i\right)^{-1}$ is a continuous function, it follows from the work of Carrillo \cite{Car99} that it fulfills the following entropy inequalities: $\forall \k\in[0,1]$, $\forall \psi\in \Dd^+(\O_i\times[0,T))$
\beqn
\lefteqn{ \int_0^T\hspace{-5pt}\int_{\O_i} | \ue - \k |Ê\partial_t \psi \, dxdt + \int_{\O_i} | u_0^\eps - \k | \psi(\cdot, 0) \, dx }\nn \\
&&+ \int_0^T\hspace{-5pt}\int_{\O_i} \left(  F_i(\ue, \k) - \eps \partial_x \left| \phii(\ue) - \phii(\k) \right| \right)\partial_x \psi \, dxdt \ge 0. \label{eq:entro_car}
\eeqn
The assertion~\eqref{eq:aeconv_1} particularly yields that for $i=1,2$ and for all $\k\in [0,1]$, 
\be\label{eq:aeconv_2}
F_i(\ue, \k) \to F_i(u,\k) \textrm{ in } L^1_{loc}(\R\times(0,T)) \textrm{ as }Ê\eps \to 0.
\ee
On the other hand, it follows from~\eqref{eq:aeconv_05} that 
\be\label{eq:aeconv_3}
\eps \partial_x | \phii(\ue) - \phii(\k)| \to 0 \textrm{ in } L^1_{loc}(\R\times(0,T)) \textrm{ as }Ê\eps \to 0.
\ee
Taking \eqref{eq:aeconv_1}, \eqref{eq:aeconv_2} and \eqref{eq:aeconv_3} 
into account in~\eqref{eq:entro_car} provides the inequality~\eqref{entro_Pih_for}. 
\vskip 5pt
The last point remaining to check is that the interface entropy condition~\eqref{eq:comp_opt} holds. Let $\left(\k^\eps\right)_\eps$ be a family of steady states to \eqref{Pe} 
converging  in $L^1_{loc}(\R)$ as $\eps$ tends to $0$ to $\k_{\rm opt}$ defined in \eqref{eq:k_opt_1}-\eqref{eq:k_opt_2}. Recall that such a family exists thanks to Proposition~\ref{steady_prop}. Then for all fixed $\eps>0$, $\k^\eps$ is a steady bounded flux solution. Hence, it follows from \eqref{bounded_comp_for} that 
\beqn
\lefteqn{ \int_0^T\hspace{-5pt}\int_\R | \ue - \k^\eps | \partial_t \psi \, dxdt + \int_\R \left| u_0^\eps - \k^\eps \right| \psi(\cdot,0) \, dxdt
}\nn \\
&& + \int_0^T \sum_{i=1,2} \int_{\O_i} \left(  F_i(\ue, \k^\eps) - \eps\partial_x \left|\phii(\ue) - \phii(\k^\eps) \right|\right)\partial_x \psi \, dxdt \ge 0. 
\label{eq:comp_ke_1}
\eeqn
It follows from Proposition~\ref{L2H1} that 
$$
 \eps\partial_x \left|\phii(\ue) - \phii(\k^\eps) \right|\to 0 \textrm{ in } L^1_{loc}(\OiT) \textrm{ as } \eps \to 0.
$$
Letting $\eps$ tend to $0$ in \eqref{eq:comp_ke_1} provides directly the fourth point~\eqref{eq:comp_opt} the fourth point in Definition~\ref{entro_Pih_def}. 
\end{proof}

\vskip 10pt
We state now the main result, which is in fact the extension of Proposition~\ref{conv_entro} to a larger 
class of initial data.
\vskip 5pt
\begin{theorem}[main result]\label{conv_final}
Let $u_0\in L^\infty(\R)$, $0\le u_0 \le 1$, and let $(\ue_0)_\eps$ a family of 
approximation of $u_0$ given by Lemma~\ref{init_lem}, 
and let ${(\ue)}_\eps$ be the induced sequence  of  bounded flux solution of \refe{Pe}. 
Then under Assumption \refe{croissance_hyp2},
$$\lim_{\eps \rightarrow 0} \ue = u \textrm{ in } L^p_{loc}(\R\times[0,T]),\quad \forall p\in[1,\infty)$$
where $u$ is the unique entropy solution to \Pih~associated to initial data $u_0$.
\end{theorem}
\vskip 5pt
\begin{proof}
Let $u_0\in L^\infty(\R)$, $0\le u_0 \le 1$, and let $\nu>0$. There exists $u_{0,\nu}$ in 
$BV(\R)$, $0\le u_{0,\nu} \le 1$ such that for all $R>0$,
\be\label{final7}
\| u_{0,\nu} - u_0 \|_{L^1(-R,R)} \le C(R)\nu.
\ee
If one regularizes $u_{0,\nu}$ into $\ue_{0,\nu}$ using Lemma~\ref{init_lem}, and if one denotes by
$\ue_\nu$ the associated unique bounded flux solution, we have seen that $\ue_\nu$ converges almost 
everywhere to $u_\nu$ as $\eps$ tends to $0$. As previously, we denote by $\ue_0$ the regularization of 
$u_0$ obtained via Lemma~\ref{init_lem}, and $\ue$ the unique associated bounded flux solution.
\vskip 5pt
 Let $R>0$, then we have
\beqn
\int_0^T\int_{-R}^R | \ue - u | dxdt & \le  & \int_0^T\int_{-R}^R | \ue - \ue_\nu | dxdt 
  + \int_0^T\int_{-R}^R | \ue_\nu - u_\nu | dxdt \nn \\
  &&+ \int_0^T\int_{-R}^R | u_\nu - u | dxdt. \label{final1}
\eeqn
The contraction principle stated in Theorem~\ref{th_Bac} yields 
\be\label{final2}
\int_0^T\int_{-R}^R | u_\nu(x,t) - u(x,t) | dxdt \le T \int_{-R-MT}^{R+MT} | u_{0,\nu}(x) - u_0(x) | dx
\ee
where $M \ge \max_{i} Lip(f_i)$. 

We denote by $\zeta(x,t) = \min(1,\left(R+1+M(T-t)-|x|\right)^+)$. One has $\zeta=1$ on $(-R,R)$, 
$\zeta\ge 0$ and $\zeta\in L^1(\R)$ thus
\beqn
\int_0^T\int_{-R}^R | \ue - \ue_\nu | dxdt & \le &  \int_0^T \int_\R  | \ue - \ue_\nu | \zeta dxdt.  \label{final4}
\eeqn
It follows from  Proposition \ref{bounded_comp} that for all $\psi\in W^{1,1}(\R\times(0,T))$ with
$\psi \ge 0$ a.e. and $\psi(\cdot,T) =0$, one has
\beqn
&\ds \int_0^T \int_\R |\ue(x,t)-\ue_\nu(x,t)|\partial_t \psi(x,t)dxdt 
+  \int_\R |\ue_0(x)-\ue_{0,\nu}(x)| \psi(x,0)dx&\nn \\
&\ds +  \int_0^T \sum_{i=1,2} \int_{\O_i} \sign(\ue(x,t)-\ue_\nu(x,t)) 
\left( f_i(\ue)(x,t)-f_i(\ue_\nu)(x,t)\right)\partial_x \psi(x,t)dxdt &\nn\\
&\ds -  \int_0^T \sum_{i=1,2} \int_{\O_i} \eps 
\partial_x|\phii(\ue)(x,t)-\phii(\ue_\nu)(x,t)| \partial_x\psi(x,t)dxdt \ge 0.&
\label{final3}
\eeqn
We denote by 
$$\Lambda_1(t) = [ -R-1 -M(T-t), -R -M(T-t) ],$$
$$\Lambda_2(t) = [ R+M(T-t), R+1 +M(T-t) ],$$
$$  \vartheta(x) =\left\{ \begin{array}{rcl} 
1&\textrm{ if }& x<0, \\
-1&\textrm{ if }& x>0,
\end{array} \right.
$$
so that 
$$ \partial_x \zeta(x,t) = \sum_{i=1,2} \vartheta(x) \chi_{\Lambda_i(t)}(x),\qquad
 \partial_t \zeta(x,t) =  -M \sum_{i=1,2}  \chi_{\Lambda_i(t)}(x).$$
Taking $\psi(x,t)=(T-t)\zeta(x,t)$ in \refe{final3} yields :
\beqn
&\ds -\int_0^T\int_\R  | \ue(x,t) - \ue_\nu(x,t) | \zeta(x,t) dxdt  + T \int_\R |\ue_0(x)-\ue_{0,\nu}(x)| \zeta(x,0)dx
&\nn\\
&\ds 
- M \int_0^T (T-t) ) \sum_{i=1,2}\int_{\Lambda_i(t)}  | \ue(x,t) - \ue_\nu(x,t) | dxdt & \nn \\
&\ds + \int_0^T (T-t) \sum_{i=1,2}\int_{\Lambda_i(t)} \sign(\ue(x,t)-\ue_\nu(x,t)) 
\left( f_i(\ue)(x,t)-f_i(\ue_\nu)(x,t)\right) dxdt   &\nn \\
&\ds - \int_0^T (T-t) \sum_{i=1,2}\int_{\Lambda_i(t)} \eps \vartheta(x)
\partial_x |\phii(\ue)(x,t) - \phii(\ue_\nu(x,t))| dxdt \ge 0.&\nn
\eeqn
Since $M \ge \max_i Lip(f_i)$, one has
\beqn
&\ds  \int_0^T (T-t) \sum_{i=1,2}\int_{\Lambda_i(t)} \sign(\ue(x,t)-\ue_\nu(x,t)) 
\left( f_i(\ue)(x,t)-f_i(\ue_\nu)(x,t)\right) dxdt   &\nn \\
&\ds \le M \int_0^T (T-t)  \sum_{i=1,2}\int_{\Lambda_i(t)}  | \ue(x,t) - \ue_\nu(x,t) | dxdt &\nn \label{final5}
\eeqn
and thus
\beqn
&\ds  \int_0^T\int_\R  | \ue(x,t) - \ue_\nu(x,t) | \zeta(x,t) dxdt  
\le T \int_\R |\ue_0(x)-\ue_{0,\nu}(x)| \zeta(x,0)dx& \nn \\ 
&\ds  - \int_0^T (T-t) \sum_{i=1,2}\int_{\Lambda_i(t)} \eps \vartheta(x)
\partial_x |\phii(\ue)(x,t) - \phii(\ue_\nu(x,t))| dxdt. &\label{final6}
\eeqn
We deduce, using \refe{final2} \refe{final4} and \refe{final6} in \refe{final1}, that
\beqn
&\ds \int_0^T\int_{-R}^R | \ue(x,t) - u(x,t) | dxdt  \le  T \int_\R |\ue_0(x) - \ue_{0,\nu}(x)| \zeta(x,0) dx & \nn\\
&\ds - \int_0^T (T-t) \sum_{i=1,2}\int_{\Lambda_i(t)} \eps \vartheta(x)
\partial_x |\phii(\ue)(x,t) - \phii(\ue_\nu(x,t))| dxdt &\nn\\
&\ds +  T \int_{-R-MT}^{R+MT} | u_{0,\nu}(x) - u_0(x) | dx 
 +  \int_0^T\int_{-R}^R | \ue_\nu(x,t) - u_\nu(x,t) | dxdt. & \label{final8}
\eeqn
We can now let $\eps$ tend to $0$. Thanks to Proposition \ref{L2H1}, we can claim 
that 
$$
\lim_{\eps\to 0}  \int_0^T (T-t) \sum_{i=1,2}\int_{\Lambda_i(t)} \eps \vartheta(x)
\partial_x |\phii(\ue)(x,t) - \phii(\ue_\nu(x,t))| dxdt = 0.
$$
We also deduce from Proposition \ref{conv_entro} that
$$
\lim_{\eps\to 0} \int_0^T\int_{-R}^R | \ue_\nu(x,t) - u_\nu(x,t) | dxdt = 0.
$$
Since $\zeta(x,0)$ is compactly supported, it follows from Lemma \ref{init_lem} and the 
dominated convergence theorem that 
$$
\lim_{\eps\to 0} \int_\R |\ue_0(x) - \ue_{0,\nu}(x)| \zeta(x,0)dx  =   
\int_\R |u_0(x) - u_{0,\nu}(x)| \zeta(x,0)dx.
$$
Thus \refe{final8} becomes:
\beqn
\limsup_{\eps\to 0} \int_0^T\int_{-R}^R | \ue(x,t) - u(x,t) | dxdt &\le& 
T \int_{-R-MT}^{R+MT} | u_{0,\nu}(x) - u_0(x) | dx \nn\\
&&+ \int_\R |u_0(x) - u_{0,\nu}(x)| \zeta(x,0)dx. \label{final9}
\eeqn
The inequality \refe{final9} holds for any $\nu>0$, and letting $\nu$ tend to $0$ leads to
$$
\lim_{\eps\to 0} \int_0^T\int_{-R}^R | \ue(x,t) - u(x,t) | dxdt = 0,
$$
then $\ue $ tends to $u$ in $L^1_{loc}(\R\times[0,T])$ as $\eps\to 0$. 
Since $0 \le \ue \le 1$ a.e., H\"older inequality gives the convergence in $L^p_{loc}(\R\times[0,T])$ 
for all finite $p$.
\end{proof}

\bibliographystyle{plain}
\bibliography{ccances}
\end{document}